\documentclass{article}

\usepackage[utf8]{inputenc}

\usepackage{parskip, amsfonts, amsmath, changepage, amsthm, setspace, geometry, verbatim, MnSymbol,chemarrow}

\usepackage[pdftex]{graphicx}
\graphicspath{ {./images/} }

\usepackage[T1]{fontenc}

\usepackage{authblk}

\usepackage{MnSymbol, wasysym, mathtools, xr}

\usepackage[tiny]{titlesec}

\usepackage{hyperref}

\newcommand*\tp[1]{\big( \begin{smallmatrix}#1\end{smallmatrix} \big)}

\newcommand*\ab[1]{\left\langle #1 \right\rangle}

\newcommand*\Mod[1]{ \; (\textup{mod} \; #1 )}

\newcommand*\tops[2]{\texorpdfstring{#1}{#2}}

\newcommand*\ol[1]{\overline{#1}}

\newcommand*{\Z}{\mathbb{Z}}
\newcommand*{\N}{\mathbb{N}}
\newcommand*{\R}{\mathbb{R}}
\newcommand*{\Q}{\mathbb{Q}}

\newcommand*\Cay{\textup{Cay}}

\renewcommand{\phi}{\varphi}

\newcommand*\Sec[1]{*{#1} \phantomsection \addcontentsline{toc}{section}{#1}}

\newtheorem{Thm}{Theorem}[section]
\newtheorem{Prop}[Thm]{Proposition}
\newtheorem{Lem}[Thm]{Lemma}

\theoremstyle{definition}

\theoremstyle{remark}

\providecommand{\keywords}[1]
{
  \small	
  \textbf{\textit{Keywords---}} #1
}

\title{A generalization of Zhu’s theorem on six-valent integer distance graphs}
\author[a]{Jonathan Cervantes}
\author[b]{Mike Krebs}
\affil[a]{University of California, Riverside, Dept. of Mathematics, Skye Hall, 900 University Ave., Riverside, CA 92521, jcerv092@ucr.edu}
\affil[b]{California State University, Los Angeles, Dept. of Mathematics, 5151 State University Drive, Los Angeles, CA 91711, mkrebs@calstatela.edu}
\date{\today}

\begin{document}

\maketitle


\begin{abstract}

Given a set $S$ of positive integers, the integer distance graph for $S$ has the set of integers as its vertex set, where two vertices are adjacent if and only if the absolute value of their difference lies in $S$.  In 2002, Zhu completely determined the chromatic number of integer distance graphs when $S$ has cardinality $3$.  Integer distance graphs can be defined equivalently as Cayley graphs on the group of integers under addition.  In a previous paper, the authors develop general methods to approach the problem of finding chromatic numbers of Cayley graphs on abelian groups.  To each such graph one associates an integer matrix.  In some cases the chromatic number can be determined directly from the matrix entries.  In particular, the authors completely determine the chromatic number whenever the matrix is of size $3\times 2$ --- precisely the size of the matrices associated to the graphs studied by Zhu.  In this paper, then, we demonstrate that Zhu's theorem can be recovered as a special case of the authors' previous results.
\end{abstract}

\section{Introduction}


An \emph{integer distance graph} is a Cayley graph on the group $\mathbb{Z}$ of integers.  In other words, given a set $S$ of positive integers, we form the graph whose vertex set is $\mathbb{Z}$ such that two vertices $x$ and $y$ are adjacent if and only if $|x-y|\in S$.  (We remark that such graphs are sometimes referred to simply as ``distance graphs'' in the literature, but as the term ``distance graph'' can refer more generally to a graph whose vertex set is a metric space with an edge between each pair of points whose distance lies in some fixed set, to avoid ambiguity we use here the term ``integer distance graph.")

Chromatic numbers of integer distance graphs have been widely investigated.  We refer the reader to \cite{Liu} for a survey of this subject and an extensive list of references.  In particular Zhu, in \cite{Zhu}, determines the chromatic number of all integer distance graphs of the form $\text{Cay}(\mathbb{Z},\{\pm a, \pm b, \pm c\})$.  Moreover, Eggleton, Erd\H{o}s, and Skilton in \cite{EES} prove that if an integer distance graph of finite degree admits a proper $k$-coloring, then it admits a periodic proper $k$-coloring.  They obtain an upper bound on the period but point out that it is quite large and very likely can be reduced considerably.

In \cite{Cervantes-Krebs-general}, the authors develop a general method for dealing with chromatic numbers of Cayley graphs of abelian groups.  In \cite{Cervantes-Krebs-small-cases} this method is summarized as follows: ``A connected Cayley graph on an abelian group with a finite generating set $S$ can be represented by its Heuberger matrix, i.e., an integer matrix whose columns generate the group of relations between members of $S$.  In a companion article, the authors lay the foundation for the use of Heuberger matrices to study chromatic numbers of abelian Cayley graphs.''  The article \cite{Cervantes-Krebs-small-cases} goes on to describe its main results: ``We call the number of rows in the Heuberger matrix the {\it dimension}, and the number of columns the {\it rank}.  In this paper,  we give precise numerical conditions that completely determine the chromatic number in all cases with dimension $1$; with rank $1$; and with dimension $\leq 3$ and rank $\leq 2$.''  Example 2.4 in \cite{Cervantes-Krebs-general} gives a general formula for a Heuberger matrix associated to an integer distance graph.  When the set $S$ of positive integers has cardinality $m$, this matrix has size $m\times (m-1)$.  The integer distance graphs in \cite{Zhu} have $|S|=3$ and thus have Heuberger matrices of dimension $3$ and rank $2$.  Hence one ought to be able to recover Zhu's theorem from the results of \cite{Cervantes-Krebs-small-cases}.

The purpose of the present paper is to do just that.  We briefly discuss the method of proof.  We begin with an integer distance graph formed by a set $S=\{a,b,c\}$ with $0<a<b<c$.  We then take the matrix of \cite[Example 2.4]{Cervantes-Krebs-general} as a starting point.  The results of \cite{Cervantes-Krebs-small-cases} require a matrix in a particular form.  In \cite{Cervantes-Krebs-general} several matrix transformations are detailed which preserve the underlying graph.  Via these transformations we morph the starting matrix into the needed form.  The next section provides full details.

Moreover, a close examination of the method of proof yields significantly improved upper bounds for the periods of optimal colorings of such graphs.

This article depends heavily on \cite{Cervantes-Krebs-general} and \cite{Cervantes-Krebs-small-cases}, which we will refer to frequently.  The reader should assume that all notation, terminology, and theorems used but not explained here are explained there.

\section{Zhu's theorem}\label{section-Zhus-theorem}


For positive integers $a,b,c$ with $\gcd(a,b,c)=1$, we define the \emph{Zhu $\{a,b,c\}$ graph} as $\text{Cay}(\mathbb{Z},\{\pm a, \pm b,\pm c\})$.  

\begin{Thm}[{\cite[Cor. 2.1]{Zhu}}]\label{theorem-Zhu}Let $X$ be a Zhu $\{a,b,c\}$ graph with $a\leq b\leq c$.  Then \[\chi(X)=\begin{cases}
2 & \text{if }a,b,c\text{ are all odd}\\
4 & \text{if }a=1, b=2,\text{ and }3|c\\
4 & \text{if }a+b=c,\text{ and }a\not\equiv b\Mod{3}\\
3 & \text{otherwise}.\end{cases}\]\end{Thm}

Note that \cite{Zhu} requires $a$, $b$, and $c$ to be distinct, but we do not.  By \cite[Example \ref{general-example-ees}]{Cervantes-Krebs-general}, Theorem \ref{theorem-Zhu} holds even when two of the integers $a,b,c$ are equal.

The most difficult part of the proof of Theorem \ref{theorem-Zhu} is showing that the upper bound of $3$ for $\chi(X)$ holds in every ``otherwise'' case.  In this section, we furnish an alternate proof of Zhu's theorem in which those $3$-colorings arise by pulling back from Heuberger circulants.  As we show, we can think about the accompanying graph homomorphisms either in terms of Heuberger matrices or more directly as reduction modulo a carefully chosen integer.

We now sketch a proof of Zhu's theorem using \cite[Theorem \ref{small-theorem-m-equals-3}]{Cervantes-Krebs-small-cases}.   Let $a_1,a_2,a_3$ be nonzero integers with $\gcd(a_1,a_2,a_3)=1$, and let $X$ be the Zhu $\{|a_1|,|a_2|,|a_3|\}$ graph.  It is straightforward to show that given any three integers, there are two of them such that either their sum or their difference is divisible by $3$.  The set $\{\pm a_1, \pm a_2, \pm a_3\}$ is unchanged if we either permute $a_1,a_2,a_3$ or else replace $a_j$ with $-a_j$.  Consequently we may assume without loss of generality that $3\mid a_1+a_2$.  Moreover, by transposing $a_1$ and $a_2$, and/or replacing $a_1$ and $a_2$ with their negatives, we may assume that $-a_1\leq a_2$ and $|a_1|\leq |a_2|$.  The purpose of these maneuvers is to find a Heuberger matrix for $X$ in modified Hermite normal form.

Let $g_2=\text{gcd}(a_1,a_2)$.  Let $u_{12}, u_{22}\in\mathbb{Z}$ such that \begin{equation}\label{equation-u12-u22}
    a_1u_{12}+a_2u_{22}=a_3g_2
\end{equation}

Recall from \cite[Example \ref{general-example-arbitrary-distance-graph}]{Cervantes-Krebs-general} that $X$ is isomorphic to \[\begin{pmatrix}
\frac{a_2}{g_2} & -u_{12} \\
-\frac{a_1}{g_2} & -u_{22} \\
0 & g_2\\
\end{pmatrix}^{\text{SACG}}\cong M^{\text{SACG}}\;\;\text{ where }\;\;M=\begin{pmatrix}
g_2 & 0\\
-u_{22} & -\frac{a_1}{g_2}\\
-u_{12} & \frac{a_2}{g_2}\\
\end{pmatrix}.\]

Observe that $M$ has no zero rows.  Moreover, note that $X$ does not have loops.  It is straightforward to show that the column sums of $M$ are both even if and only if $a_1, a_2$, and $a_3$ are all odd, so by \cite[Lemma \ref{general-lemma-bipartite}]{Cervantes-Krebs-general} we have that $\chi(X)=2$ if and only if $a_1, a_2$, and $a_3$ are all odd.  Assume now that $\chi(X)\neq 2$.  Thus by \cite[Lemma \ref{small-lemma-modified-Hermite-normal-form}]{Cervantes-Krebs-small-cases} and \cite[Theorem \ref{small-theorem-m-equals-3}]{Cervantes-Krebs-small-cases} we have that $\chi(X)$ equals either $3$ or $4$.     Take $a,b,c$ so that $0<a\leq b\leq c$ and $\{a,b,c\}=\{|a_1|,|a_2|,|a_3|\}$.  It remains to show that $\chi(X)=4$ if and only if either (i) $a=1, b=2,\text{ and }3|c$, or else (ii) $a+b=c,\text{ and }a\not\equiv b\Mod{3}$.

If (i) holds, then by \cite[Example \ref{general-example-arbitrary-distance-graph}]{Cervantes-Krebs-general} we have that $X$ is isomorphic to  \(\begin{pmatrix}
1 & 0\\
0 & -1\\
3(c/3) & 2
\end{pmatrix}^{\text{SACG}}\), whence we have $\chi(X)=4$ by \cite[Theorem \ref{small-theorem-m-equals-3}]{Cervantes-Krebs-small-cases}.

If (ii) holds, then either $a$ or $b$ or $c$ must be divisible by $3$.  If, say, $3\mid b$, then by \cite[Example \ref{general-example-arbitrary-distance-graph}]{Cervantes-Krebs-general} we have that $X$ is isomorphic to  \(\begin{pmatrix}
1 & 0\\
-1 & a\\
-1 & a+3(k-1)
\end{pmatrix}^{\text{SACG}}\) with $k=(b+3)/3$, whence we have $\chi(X)=4$ by \cite[Theorem \ref{small-theorem-m-equals-3}]{Cervantes-Krebs-small-cases}.  Similar arguments give us $\chi(X)=4$ when $3\mid a$ or $3\mid c$.

Conversely, suppose that $\chi(X)=4$, and we will show that either (i) or (ii) is satisfied.  To apply \cite[Theorem \ref{small-theorem-m-equals-3}]{Cervantes-Krebs-small-cases}, we must first put $M$ in modified Hermite normal form.  The matrix $M$ satisfies all conditions of \cite[Def. \ref{small-def-modified-Hermite-normal-form}]{Cervantes-Krebs-small-cases} except the last; this can be rectified with help from the division theorem.  Let $q$ and $r$ be integers such that\begin{equation}\label{equation-q-r}
    -u_{22}=q\left(-\frac{a_1}{g_2}\right)+r, \text{ where }-\left|\frac{a_1}{g_2}\right|<r\leq 0.
\end{equation}We assume now that $-\left|\frac{a_1}{2g_2}\right|\leq r$, and we leave to the reader the other, similar case where this inequality does not hold.  Adding $-q$ times the second column of $M$ to the first, we obtain the matrix \[M_1=\begin{pmatrix}
g_2 & 0\\
r & -\frac{a_1}{g_2}\\
-u_{12}-\frac{q a_2}{g_2} & \frac{a_2}{g_2}
\end{pmatrix}.\]We have that $X$ is isomorphic to $M^{\text{SACG}}_1$ and that $M_1$ is in modified Hermite normal form.  Thus $M_1$ equals one of the six types of matrices listed in the third statement in \cite[Theorem \ref{small-theorem-m-equals-3}]{Cervantes-Krebs-small-cases}.  We discuss here only the case where \[M_1=\begin{pmatrix}
1 & 0\\
0 & 1\\
3k & 1+3k
\end{pmatrix}\]for some positive integer $k$, and leave the other five cases for the reader.  In this case we have $g_2=1$, $r=0$, $a_1=-1$, $a_2=1+3k$, and $-u_{12}-q a_2=3k$.  From (\ref{equation-q-r}) we get that $u_{22}=-q$.  So by (\ref{equation-u12-u22}) we get that $a_3=3k$.  From this we see that $a=1$ and $b=3k$ and $c=1+3k$, so condition (ii) is met.\hfill$\square$



\vspace{.1in}

We have natural graph homomorphisms from Zhu graphs to Heuberger circulants given by reducing modulo an appropriate integer.  The next lemma recasts these homomorphisms in terms of Heuberger matrices associated to the corresponding standardized abelian Cayley graphs.  
\begin{Lem}\label{lemma-homomorphism-Zhu-to-Heuberger}Let $a,b,c$ be positive integers such that $\gcd(a,b,c)=1$ and $b+c\nmid a$.  Then $C_{b+c}(a,b)$ is a Heuberger circulant graph.  Moreover, let $X$ and $Y$ be standardized Cayley graphs defined by \[\begin{pmatrix}
y_{11} & y_{12}\\
y_{21} & y_{22}\\
y_{31} & y_{32}
\end{pmatrix}^{\text{SACG}}_X\text{ and }\begin{pmatrix}
y_{11} & y_{12}\\
y_{21}-y_{31} & y_{22}-y_{32}
\end{pmatrix}^{\text{SACG}}_Y.\]  Suppose we have an isomorphism between the Zhu $\{a,b,c\}$ graph and $X$ given by the map $\phi_X\colon \mathbb{Z}^3\to\mathbb{Z}$ defined by $\phi_X\colon e_1\mapsto a, e_2\mapsto b, e_3\mapsto c$.  Then $\phi_Y\colon \mathbb{Z}^2\to\mathbb{Z}_{b+c}$ defined by $\phi_Y\colon e_1\mapsto a, e_2\mapsto b$ gives us an isomorphism between $Y$ and $C_{b+c}(a,b)$.  Furthermore, the following diagram of graph homomorphisms commutes, where $\tau_1$ is defined by $e_1\mapsto e_1, e_2\mapsto e_2, e_3\mapsto -e_2$, and $\tau_2$ is defined by reduction modulo $b+c$.\[\begin{array}{ccl}
\begin{pmatrix}
y_{11} & y_{12}\\
y_{21} & y_{22}\\
y_{31} & y_{32}
\end{pmatrix}^{\text{SACG}}_X & \xrightarrow[\tau_1]{\ocirc} & \begin{pmatrix}
y_{11} & y_{12}\\
y_{21}-y_{31} & y_{22}-y_{32}
\end{pmatrix}^{\text{SACG}}_Y\\
\phi_X\Bigg\downarrow\qquad & \; & \qquad\Bigg\downarrow\phi_Y\\
\text{Cay}(\mathbb{Z},\{\pm a,\pm b,\pm c\}) & \xrightarrow[\tau_2]{\;} & C_{b+c}(a,b)
\end{array}\]\end{Lem}\begin{proof}The conditions $\gcd(a,b,c)=1$ and $b+c\nmid a$ guarantee that $C_{b+c}(a,b)$ meets the criteria of \cite[Def. \ref{small-def-Heuberger-circulant}]{Cervantes-Krebs-small-cases}.

Let $M_X$ and $M_Y$, respectively, be the above matrices defining the graphs $X$ and $Y$.  Using the fact that the kernel of $\phi_X$ equals the $\mathbb{Z}$-span of the columns of $M_X$, it is then a routine exercise to show that the kernel of $\phi_Y$ equals the $\mathbb{Z}$-span of the columns of $M_Y$, whence it follows that $\phi_Y$ is an isomorphism.

Finally, we have that $\tau_2\circ\phi_X(e_j)=\phi_Y\circ\tau_1(e_j)$ for $j=1,2,3$; hence the diagram is commutative.\end{proof}

In a nutshell: To reduce the Zhu $\{a,b,c\}$ graph modulo $b+c$, we subtract the third row from the second row to obtain the Heuberger circulant $C_{b+c}(a,b)$.  Indeed, the proof of Lemma \ref{lemma-homomorphism-Zhu-to-Heuberger} generalizes in a similar fashion to any number of variables.

Of course, we can play the same game with any pair of variables in lieu of $b$ and $c$.  Moreoever, we can reduce by $b-c$ instead of $b+c$ by adding the two rows instead of subtracting them.


In \cite{EES} a \emph{periodic $k$-coloring} of an integer distance graph $X$ with \emph{period} $p$ is a $k$-coloring $c$ of $X$ such that $c(n)=c(n+p)$ for all $n\in\mathbb{Z}$.  Equivalently, a  periodic $k$-coloring of an integer distance graph $X$ with period $p$ is a pullback, via the map $n\mapsto\overline{n}$, of a $k$-coloring of a circulant graph of order $p$.  It is proved in \cite{EES} that if an integer distance graph $\text{Cay}(\mathbb{Z},S)$, where $S$ is finite, has chromatic number $\chi$, then it has a proper periodic $\chi$-coloring.  That article provides what the authors describe as an ``explicit (but weak)'' upper bound of $qk^q$ for the smallest period for such colorings, where $q=\max S$.

We now show that for a Zhu $\{a,b,c\}$ graph, we can indeed improve this upper bound considerably.  For we have just shown that Theorem \ref{theorem-Zhu} follows from \cite[Theorem \ref{small-theorem-m-equals-3}]{Cervantes-Krebs-small-cases}.  In the proof of \cite[Theorem \ref{small-theorem-m-equals-3}]{Cervantes-Krebs-small-cases}, all colorings are constructed via homomorphisms obtained by collapsing two rows by adding or subtracting them.  By Lemma \ref{lemma-homomorphism-Zhu-to-Heuberger}, such a homomorphism corresponds to reduction modulo the sum or difference of two of $a,b,c$.  Thus we have the following proposition.

\begin{Prop}Let $0<a\leq b\leq c$ be integers.  Suppose the Zhu $\{a,b,c\}$ graph has chromatic number $\chi$.  Then it admits a periodic $\chi$-coloring with period $\leq b+c$.\end{Prop}

\section*{Acknowledgments}

The authors wish to thank Daphne Liu for her insights into the history of Zhu's theorem.

\bibliographystyle{amsplain}
\bibliography{references}



\end{document}